\documentclass[dvipdfmx]{amsart}
\usepackage{amsfonts}
\usepackage{amssymb}
\usepackage{amsmath}
\usepackage{amsthm}
\usepackage{amscd}
\usepackage{mathrsfs}
\usepackage{graphicx, color}
\usepackage{url}

\theoremstyle{theorem}
\newtheorem{theorem}{Theorem}[section]

\theoremstyle{definition}
\newtheorem{definition}[theorem]{Definition}
\newtheorem{example}[theorem]{Example}

\newtheorem{remark}[theorem]{Remark}

\newcommand{\Int}[1]{
 {\mathrm{Int}(#1)}
}

\begin{document}

\title{Alteration of Seifert surfaces}
\author{Ayumu Inoue}
\address{Department of Mathematics, Tsuda University, 2-1-1 Tsuda-machi, Kodaira-shi, Tokyo 187-8577, Japan}
\email{ayminoue@tsuda.ac.jp}


\begin{abstract}
We introduce the notion of alteration of a surface embedded in a 3-manifold extending that of compression.
We see that given two Seifert surfaces of the same link are related to each other by ``single'' alteration, even if they are not by compression.
\end{abstract}

\maketitle

\section{Introduction}
\label{sec:introduction}

Relationships between Seifert surfaces of the same link have been studied from various points of view.
It is well-known that given two Seifert surfaces of the same link are tube equivalent, i.e., they are related up to ambient isotopy by the addition or removal of tubes.
Kakimizu \cite{Kak1992} introduced a complex for a non-split link, of which the vertices correspond to the isotopy class of the minimal Seifert surfaces of the link and vertices span a simplex if they have mutually disjoint representatives in the exterior of the link.
He showed that this complex is connected, extending the work of Scharlemann and Tompson \cite{ST1988} for knots.
Schultens \cite{Sch2010} showed that the Kakimizu complex of a knot is simply connected, then Przytycki and Schultens \cite{PS2012} extended the claim for non-split links.
Several authors (Schaufele \cite{Sch1967} is the first one as far as the author know) gave us Seifert surfaces of links which are not minimal but incompressible.
Kobayashi \cite{Kob1989} introduced a genus three Seifert surface of the trivial knot, depicted in Figure \ref{fig:trivial_knot_orig}, which has a funny property on compression:
Although we can compress it to a minimal one (i.e., a disk) by a sequence of compression, this Seifert surface has no triple of mutually disjoint compressing disks along which we obtain a minimal one by compression.
\begin{figure}[htbp]
 \centering
 \includegraphics[scale=0.20]{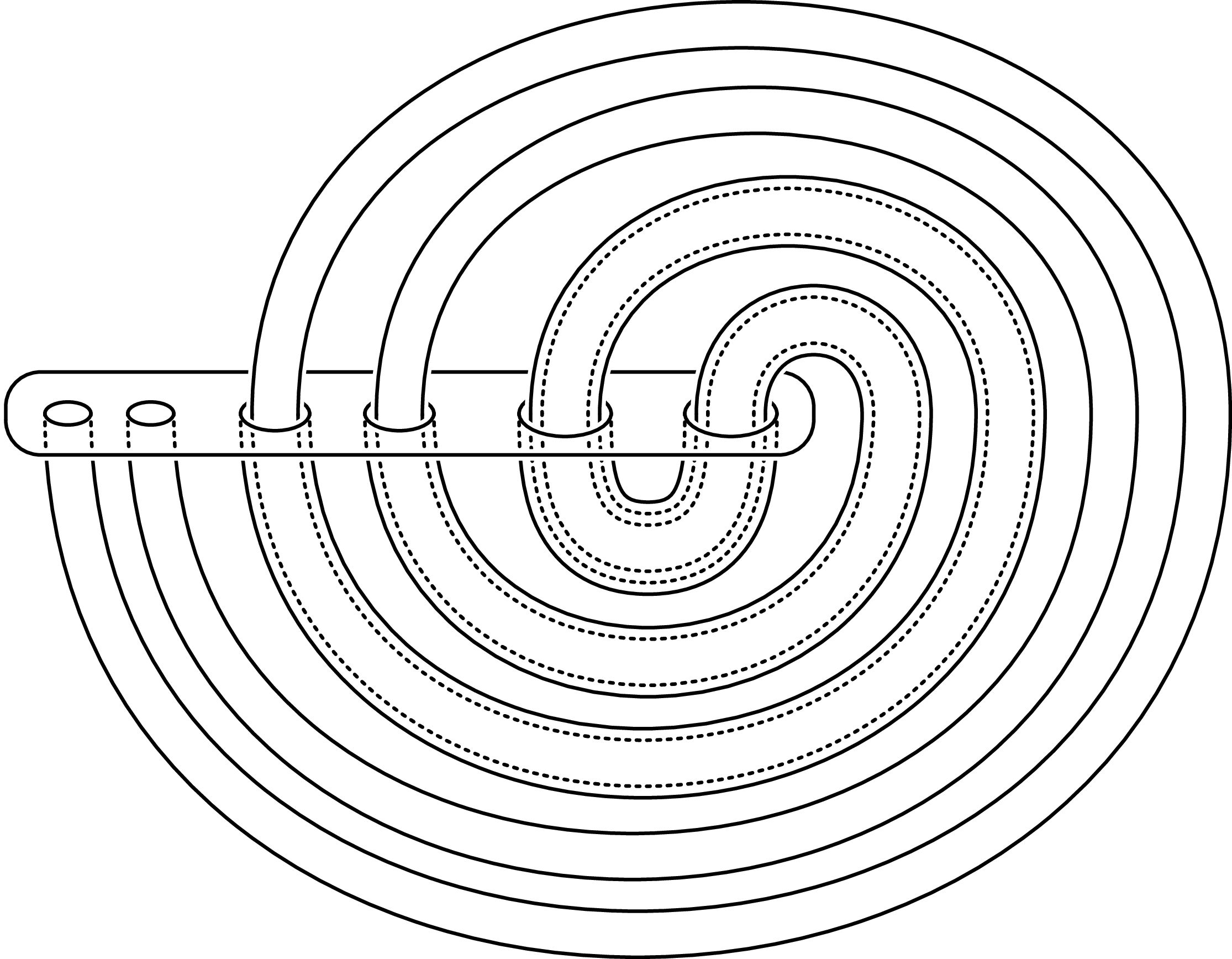}
 \caption{Kobayashi's Seifert surface of the trivial knot}
 \label{fig:trivial_knot_orig}
\end{figure}

In the studies, surgery of surfaces (e.g., compression, tubing, cut-and-paste) has often been playing key roles.
The aim of this paper is to introduce the following surgery method of surfaces and to examine how it works on Seifert surfaces:

\begin{definition}
Let $S$ be an oriented surface, which is not necessarily connected or closed, embedded in a 3-manifold $M$.
A connected orientable surface $F$ with non-empty boundary embedded in $M$ is said to be an \emph{altering surface} for $S$ if we can thicken $F$ to $F \times [-1, 1]$ in $M$ so that it satisfies the following conditions:
\begin{itemize}
\item
$F \times \{ 0 \} = F$
\item
$S \cap (F \times [-1, 1]) = \partial F \times [-1, 1] \subset \Int{S}$
\item
The orientation of $S \setminus (\partial F \times (-1, 1))$ is consistently extended to that of the surface $T = (S \cup \partial (F \times [-1, 1])) \setminus (\partial F \times (-1, 1))$ embedded in $M$
\end{itemize}
We say that $T$ is obtained from $S$ by \emph{alteration} along $F$.
\end{definition}

A compressing disk $D$ for $S$ is a typical altering surface, and alteration of $S$ along $D$ is nothing less than compression of $S$ along $D$.
We note that alteration might increase the first Betti number of surfaces, while compression certainly does not.
In contrast with behavior of compression, mentioned above, we have the following theorem:

\begin{theorem}
\label{thm:main1}
For given Seifert surfaces $S$ and $T$ of the same link, there are mutually disjoint altering surfaces for $S$ along which we obtain $T$ up to ambient isotopy from $S$ by sequential alteration removing the closed components.
\end{theorem}

We prove the theorem in Section \ref{sec:proof_of_theorem1}.
We further see several altering surfaces for Seifert surfaces, in Section \ref{sec:examples}, which demonstrate the difference between compression and alteration.

Throughout this paper, each link and its Seifert surfaces are assumed to be oriented consistently and to lie in the three sphere.
Although each Seifert surface of a link is allowed to be disconnected, it is not allowed to have closed components.

\section{Quick review on tube equivalence}
\label{sec:quick_review_on_tube_equivalence}

We prove Theorem \ref{thm:main1} in light of the fact that given two Seifert surfaces of the same link are tube equivalent.
We thus recall the fact briefly in this section.

Let $S$ be a surface, which is not necessarily connected or closed, embedded in a 3-manifold $M$.
A 3-ball $V = D^{2} \times D^{1}$ (or $V = D^{1} \times D^{2}$) embedded in $M$ is said to be a \emph{1-handle} (or a \emph{2-handle}) attaching to $S$ if the intersection of $S$ and $V$ coincides with $D^{2} \times \partial D^{1}$ (or $D^{1} \times \partial D^{2}$) and lies in $\Int{S}$.
Associated with a 1- or 2-handle $V$ attaching to $S$, we have the surface $T = (S \cup \partial V) \setminus \Int{S \cap V}$ embedded in $M$.
We say that $T$ is obtained from $S$ by \emph{surgery} along $V$.
Further assume that $S$ is oriented.
Then $V$ is said to be \emph{coherent} if the orientation of $S \setminus \Int{S \cap V}$ is consistently extended to that of $T$.

Two oriented surfaces $S$ and $T$ embedded in $M$ are said to be \emph{tube equivalent} if there is a finite sequence
\[
 S = S_{0} \to S_{1} \to S_{2} \to \dots \to S_{N} = T \qquad (N \geq 0)
\]
of oriented surfaces $S_{i}$ embedded in $M$ in which $S_{i+1}$ is obtained from $S_{i}$ up to ambient isotopy by surgery along a coherent 1- or 2-handle attaching to $S_{i}$.
We note that surgery of a surface along a 1- or 2-handle is respectively regarded as addition or removal of a tube to or from the surface.
The following theorem is known well:

\begin{theorem}
\label{thm:tube_equivarent}
Given two Seifert surfaces of the same link are tube equivalent.
\end{theorem}

We refer the reader to \cite{BFK1998} for an elementary but elegant proof of the theorem.

\section{Proof of Theorem \ref{thm:main1}}
\label{sec:proof_of_theorem1}

We devote this section to prove Theorem \ref{thm:main1}.
We start with preparing some terminologies.

Let $S$ be an oriented surface embedded in a 3-manifold, which is not necessarily connected or closed.
Suppose that $F_{1}, F_{2}, \dots, F_{n}$ are mutually disjoint altering surfaces for $S$ ($n \geq 1$).
Without loss of generality, we may assume that $F_{i} \times [-1, 1]$ is also disjoint from $F_{j} \times [-1, 1]$ for each $i$ and $j$ ($1 \leq i < j \leq n$).
Then we call $\mathscr{F} = \{ F_{1}, F_{2}, \dots, F_{n} \}$ a \emph{system of altering surfaces} (\emph{SAS} for short) for $S$.
We let alteration of $S$ along a SAS $\mathscr{F}$ mean sequential alteration of $S$ along the altering surfaces in $\mathscr{F}$.
We further allow a SAS to be the empty set.
Alteration of $S$ along the empty set does not change $S$ at all.

Suppose that $S$ is a Seifert surface of a link and $\mathscr{F}$ a SAS for $S$.
In this case, we obtain the union of a Seifert surface $T^{\prime}$ of the link and some (or no) closed surfaces from $S$ by alteration along $\mathscr{F}$.
Since we are mainly interested in Seifert surfaces, in this paper, we say in this situation that $\mathscr{F}$ \emph{yields} $T^{\prime}$ for abbreviation.

With the terminologies, Theorem \ref{thm:main1} is restated as follows:

\begin{theorem}[restatement of Theorem \ref{thm:main1}]
\label{thm:main1'}
For given Seifert surfaces $S$ and $T$ of the same link, there is a SAS for $S$ which yields $T$ up to ambient isotopy.
\end{theorem}

\begin{proof}
In light of Theorem \ref{thm:tube_equivarent}, we have a finite sequence
\[
 T = T_{0} \to T_{1} \to T_{2} \to \dots \to T_{N} = S \qquad (N \geq 0)
\]
of Seifert surfaces $T_{i}$ of the link in which $T_{i+1}$ is obtained from $T_{i}$ up to ambient isotopy by surgery along a coherent 1- or 2-handle attaching to $T_{i}$.
We inductively see that the following claim is true for each $i$ ($0 \leq i \leq N$):
\begin{itemize}
\item[($\clubsuit$)]
There is a SAS $\mathscr{F}_{i}$ for $T_{i}$ which yields $T$ up to ambient isotopy.
\end{itemize}
Since $T_{0}$ coincides with $T$, ($\clubsuit$) is obviously true for $i = 0$ letting $\mathscr{F}_{0}$ to be the empty set.
In what follows, we assume that ($\clubsuit$) is true for $i = j$.

Suppose that $T_{j+1}$ is obtained from $T_{j}$ up to ambient isotopy by surgery along a coherent 1-handle $V$ attaching to $T_{j}$.
In this case, without loss of generality, we may assume that the union of the altering surfaces in $\mathscr{F}_{j}$ intersects with $V$ at mutually disjoint 2-disks $D^{2} \times \{ p_{1}, p_{2}, \dots, p_{k} \}$, as depicted in Figure \ref{fig:1-handle_1}.
Here, $k \geq 0$ and $p_{1}, p_{2}, \dots, p_{k}$ are points in $D^{1} = [-1, 1]$ satisfying $-1 < p_{1} < p_{2} < \dots < p_{k} < 1$.
\begin{figure}[htbp]
 \centering
 \includegraphics[scale=0.15]{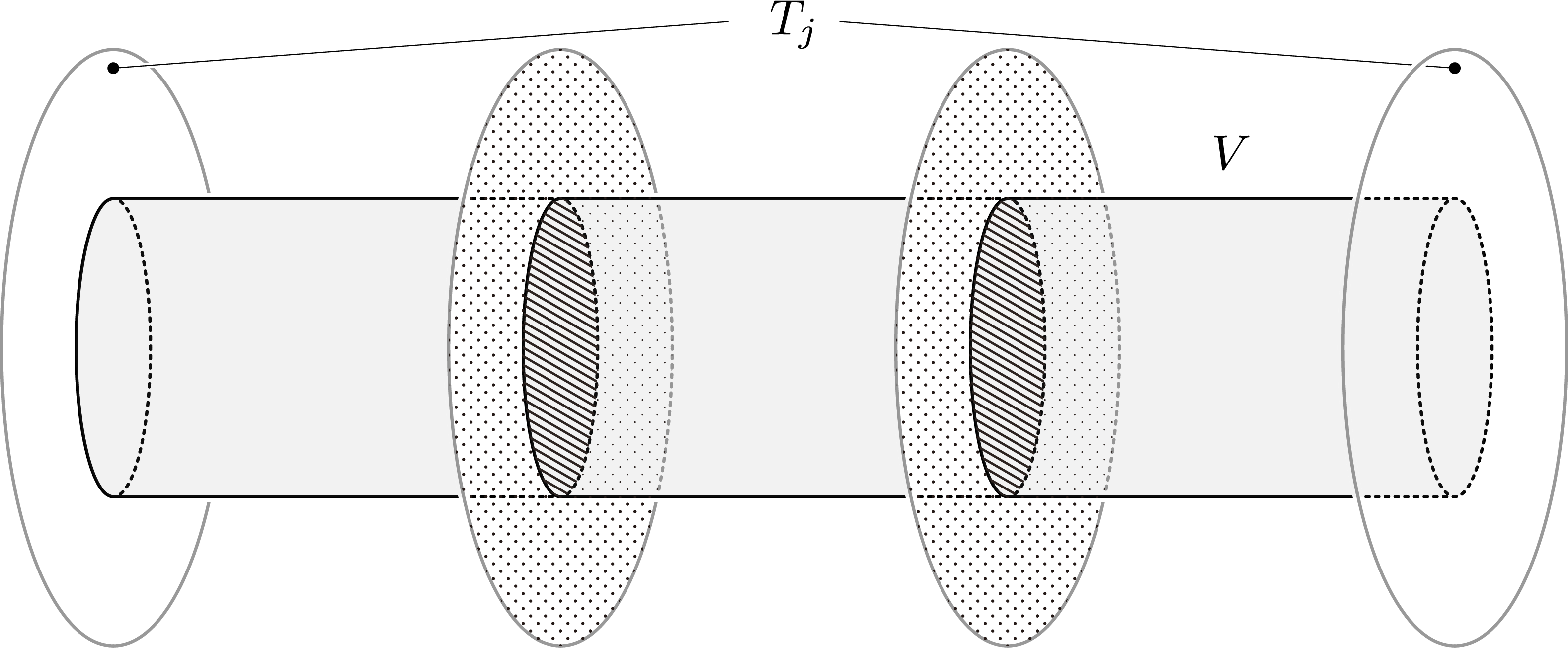}
 \caption{Some altering surfaces (dotted ones) in $\mathscr{F}_{j}$ intersect with $V$ at 2-disks (shaded ones)}
 \label{fig:1-handle_1}
\end{figure}

We construct a SAS $\mathscr{F}_{j+1}$ for $T_{j+1}$ as follows.
We first add each altering surface in $\mathscr{F}_{j}$ into $\mathscr{F}_{j+1}$ after removing the interior of its intersection with $V$.
We next add 2-disks $D^{2} \times \{ (p_{u-1} + p_{u}) / 2 \}$ ($1 \leq u \leq k+1$) into $\mathscr{F}_{j+1}$ letting $p_{0} = -1$ and $p_{k+1} = 1$.
Then, as illustrated in Figure \ref{fig:1-handle_2}, $\mathscr{F}_{j}$ and $\mathscr{F}_{j+1}$ yield the same Seifert surface up to ambient isotopy.
Thus ($\clubsuit$) is also true for $i = j + 1$.
\begin{figure}[htbp]
 \centering
 \includegraphics[scale=0.15]{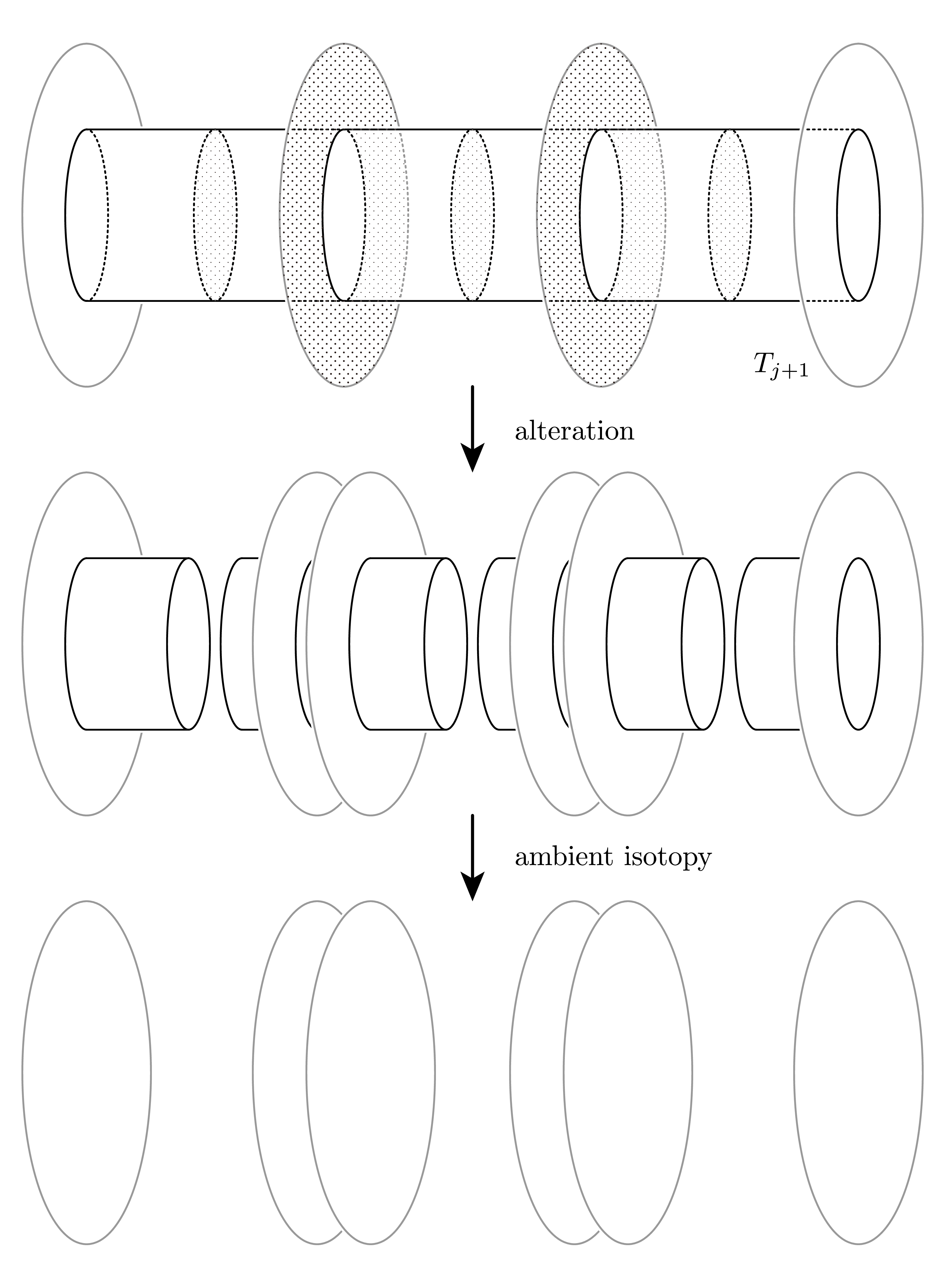}
 \caption{Alteration of $T_{j+1}$ along $\mathscr{F}_{j+1}$ (including dotted ones on the top) brings us a surface which is ambient isotopic to the surface (depicted in the bottom) obtained from $T_{j}$ by alteration along $\mathscr{F}_{j}$}
 \label{fig:1-handle_2}
\end{figure}

Suppose that $T_{j+1}$ is obtained from $T_{j}$ up to ambient isotopy by surgery along a coherent 2-handle $V$ attaching to $T_{j}$.
In this case, without loss of generality, we may assume that the union of the altering surfaces in $\mathscr{F}_{j}$ intersects with $V$ at mutually disjoint line segments $D^{1} \times \{ p_{1}, p_{2}, \dots, p_{k} \}$, 2-disks $D^{1} \times \{ \alpha_{1}, \alpha_{2}, \dots, \alpha_{l} \}$, and annuli $D^{1} \times \{ \beta_{1}, \beta_{2}, \dots, \beta_{m} \}$, as depicted in Figure \ref{fig:2-handle_1}.
Here, $k, l, m \geq 0$, $p_{1}, p_{2}, \dots, p_{k}$ are points on $\partial D^{2}$, $\alpha_{1}, \alpha_{2}, \dots, \alpha_{l}$ properly embedded arcs in $D^{2}$, and $\beta_{1}, \beta_{2}, \dots, \beta_{m}$ embedded circles in $\Int{D^{2}}$.
\begin{figure}[htbp]
 \centering
 \includegraphics[scale=0.15]{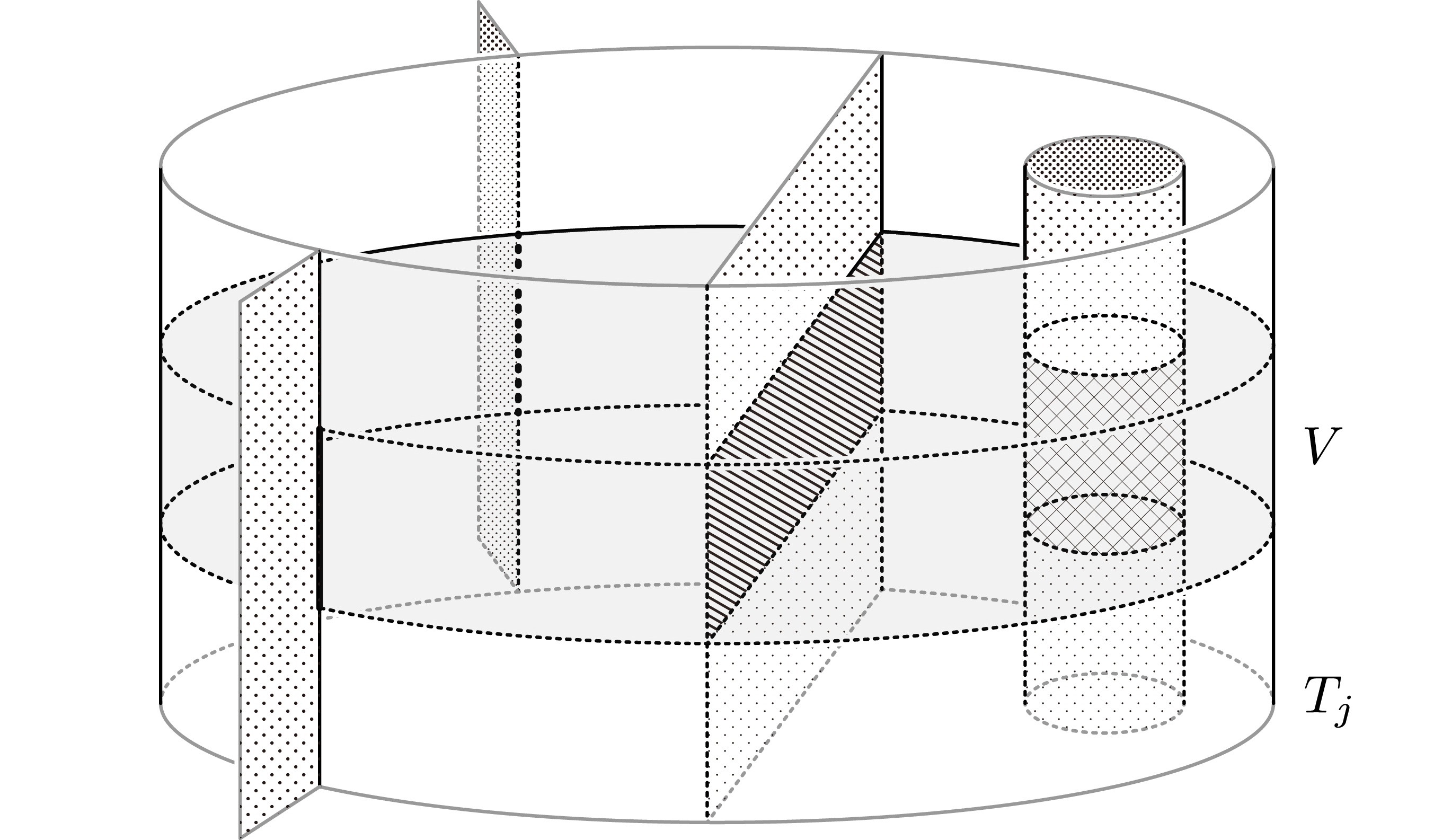}
 \caption{Some altering surfaces (dotted ones) in $\mathscr{F}_{j}$ intersect with $V$ at line segments (thicken ones), 2-disks (shaded ones), or annuli (checked ones)}
 \label{fig:2-handle_1}
\end{figure}

To construct a SAS $\mathscr{F}_{j+1}$ for $T_{j+1}$, we start with preparing the following parts.
Let $A_{1}, A_{2}, \dots, A_{r}$ and $B_{1}, B_{2}, \dots, B_{s}$ be respectively the annuli and bands each of which is a connected component of the subset
\begin{align*}
 & \left( \left( D^{1} \times \partial D^{2} \right) \cup \left( \bigcup_{u = 1}^{l} (D^{1} \times \alpha_{u}) \times \{ \pm 1 \} \right) \right) \\
 & \qquad \setminus \left( \left( \bigcup_{v = 1}^{k} (D^{1} \times \{ p_{v} \}) \times \left( - \frac{1}{2}, \frac{1}{2} \right) \right) \cup \left( \bigcup_{w = 1}^{l} (D^{1} \times \partial \alpha_{w}) \times (-1, 1) \right) \right)
\end{align*}
of $\partial V$.
Furthermore, for each circle $\beta_{u}$, we let $C_{2 u - 1} = (D^{1} \times \beta_{u}) \times \{ -1 \}$ and $C_{2 u} = (D^{1} \times \beta_{u}) \times \{ 1 \}$ be annuli.
Figure \ref{fig:2-handle_2} depicts those two types of annuli and bands.
We note that we only have an annulus $A_{1} = D^{1} \times \partial D^{2}$ if $k = l = m = 0$.
\begin{figure}[htbp]
 \centering
 \includegraphics[scale=0.15]{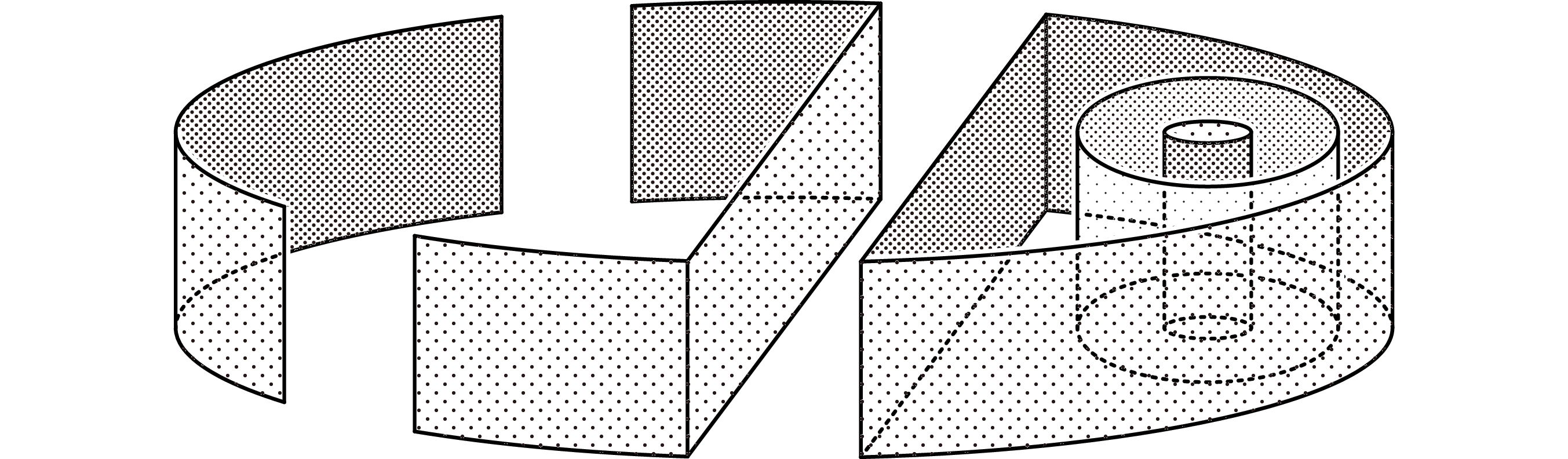}
 \caption{Two types of annuli and bands}
 \label{fig:2-handle_2}
\end{figure}

We now construct a SAS $\mathscr{F}_{j+1}$ for $T_{j+1}$ as follows.
Suppose that $F_{1}, F_{2}, \dots, F_{n}$ are the altering surfaces in $\mathscr{F}_{j}$ each of which intersects with $V$ at least once.
We first add the altering surfaces in $\mathscr{F}_{j} \setminus \{ F_{1}, F_{2}, \dots, F_{n} \}$ into $\mathscr{F}_{j+1}$.
We next add the connected components of
\[
 \left( \left( \bigcup_{u = 1}^{n} F_{u} \times \left\{ \pm \frac{1}{2} \right\} \right) \setminus \left( \Int{D^{1}} \times D^{2} \right) \right) \cup \left( \bigcup_{v = 1}^{s} B_{v} \right)
\]
into $\mathscr{F}_{j+1}$.
We finally add the annuli $A_{1}, A_{2}, \dots, A_{r}$ and $C_{1}, C_{2}, \dots, C_{2 m}$ into $\mathscr{F}_{j+1}$.
Then, as illustrated in Figure \ref{fig:2-handle_3}, $\mathscr{F}_{j}$ and $\mathscr{F}_{j+1}$ yield the same Seifert surface up to ambient isotopy.
Thus ($\clubsuit$) is also true for $i = j + 1$.
\begin{figure}[htbp]
 \centering
 \includegraphics[scale=0.15]{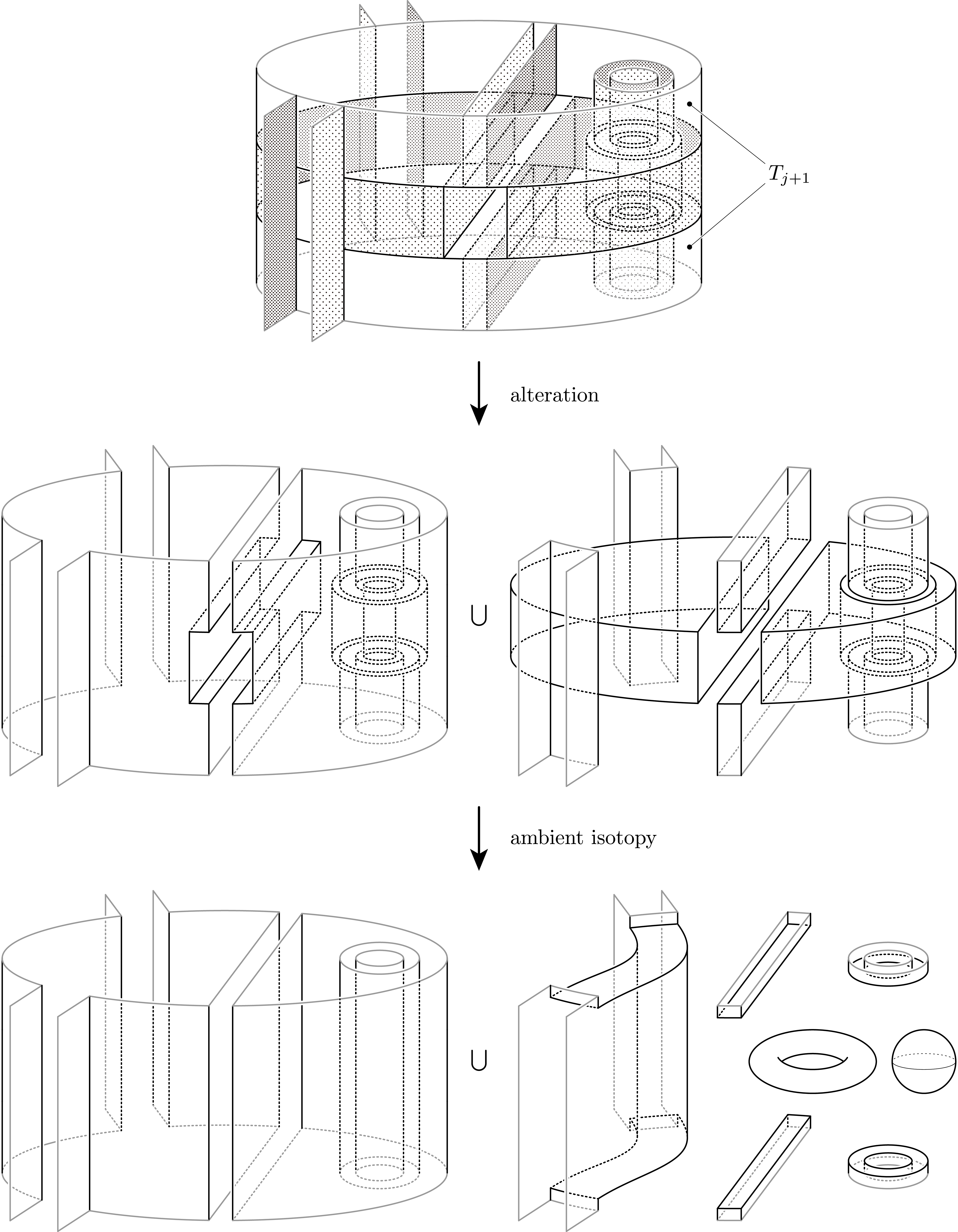}
 \caption{Alteration of $T_{j+1}$ along $\mathscr{F}_{j+1}$ (including dotted ones on the top) brings us a surface which is ambient isotopic to the union of the surface (depicted in the left-hand side of the bottom) obtained from $T_{j}$ by alteration along $\mathscr{F}_{j}$ and some closed surfaces (depicted in the right-hand side of the bottom)}
 \label{fig:2-handle_3}
\end{figure}
\end{proof}

\begin{remark}
The SAS $\mathscr{F}_{i}$, which we have constructed in the above proof, may not be ``minimal'' for $T_{i}$ to yield $T$.
For example, assume that $T_{i}$ is obtained from $T_{i-1}$ up to ambient isotopy by surgery along a coherent 2-handle $V$ attaching $T_{i-1}$, and $F$ is the only altering surface in $\mathscr{F}_{i-1}$ intersecting with $V$.
We further assume that $F \cap V$ is an annulus $D^{1} \times \beta$.
Then the SAS
\[
 \mathscr{F}_{i}^{\prime} = \left( \mathscr{F}_{i-1} \setminus \{ F \} \right) \cup \{ F \setminus \Int{V}, \, (D^{1} \times \beta) \times \{ \pm 1 \}\}
\]
for $T_{i}$ also yields $T$, although the cardinality of $\mathscr{F}_{i}^{\prime}$ is less than $\mathscr{F}_{i}$.
\end{remark}

\section{Examples}
\label{sec:examples}

We conclude the paper seeing examples of concrete altering surfaces for Seifert surfaces.
They emphasize the difference between compression and alteration.

\begin{example}
\label{eg:not_minimal}
It was proved first by Parris \cite{Par1978} and later by Oertel \cite{Oer1980} (as mentioned in Gabai's paper \cite{Gab1983}) that the Seifert surface, depicted in Figure \ref{fig:1_m5_m5_m1_5_5_orig}, of a pretzel link is not minimal but incompressible.
On the other hand, it is routine to see that the altering surface, depicted in Figure \ref{fig:1_m5_m5_m1_5_5_altering_surface}, for the Seifert surface yields a minimal Seifert surface of the link.
We note that this altering surface is homeomorphic to an annulus.
\begin{figure}[htbp]
 \centering
 \includegraphics[scale=0.15]{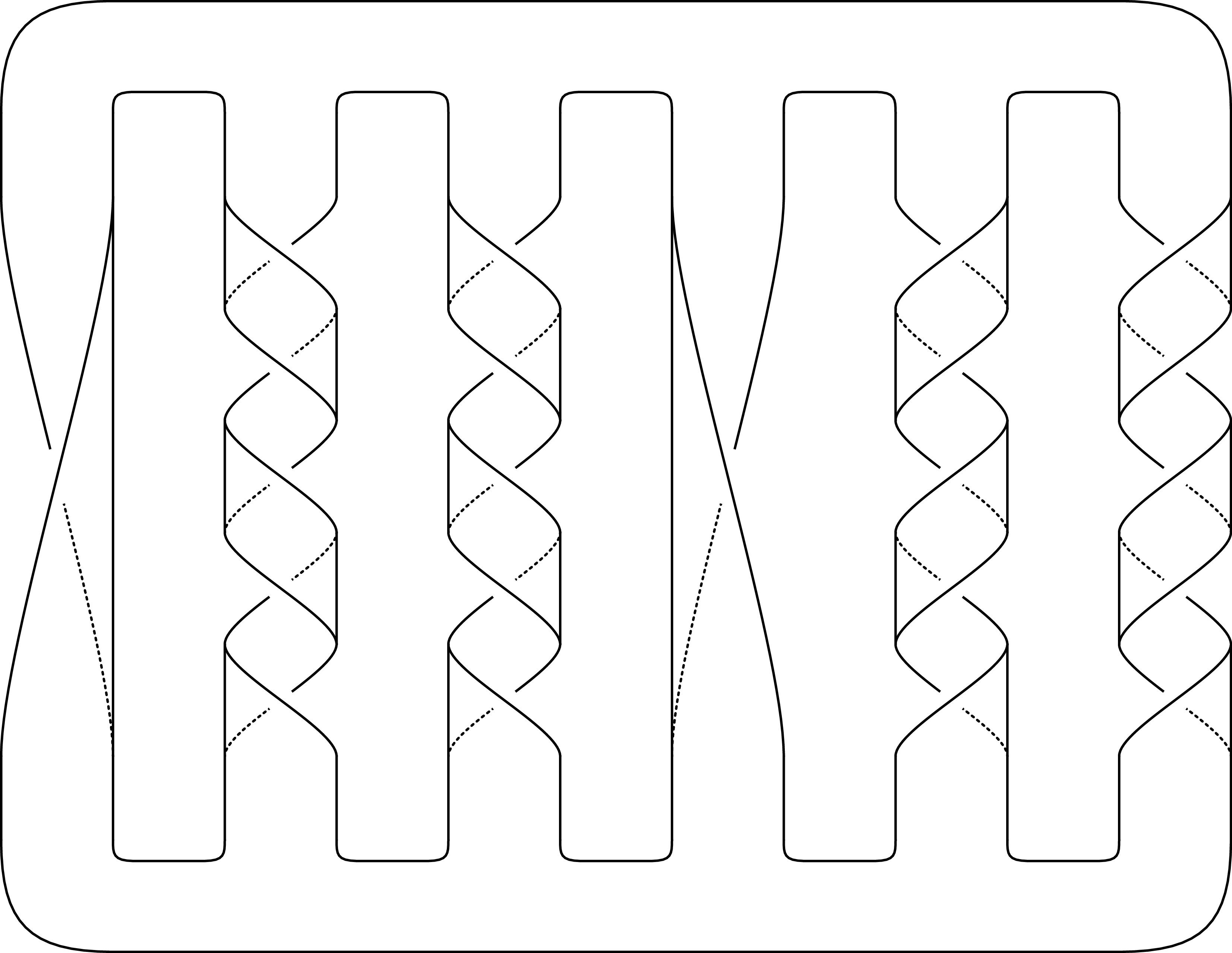}
 \caption{A Seifert surfaces of the $(1, -5, -5, -1, 5, 5)$-pretzel link, which is not minimal but incompressible}
 \label{fig:1_m5_m5_m1_5_5_orig}
\end{figure}
\begin{figure}[htbp]
 \centering
 \includegraphics[scale=0.15]{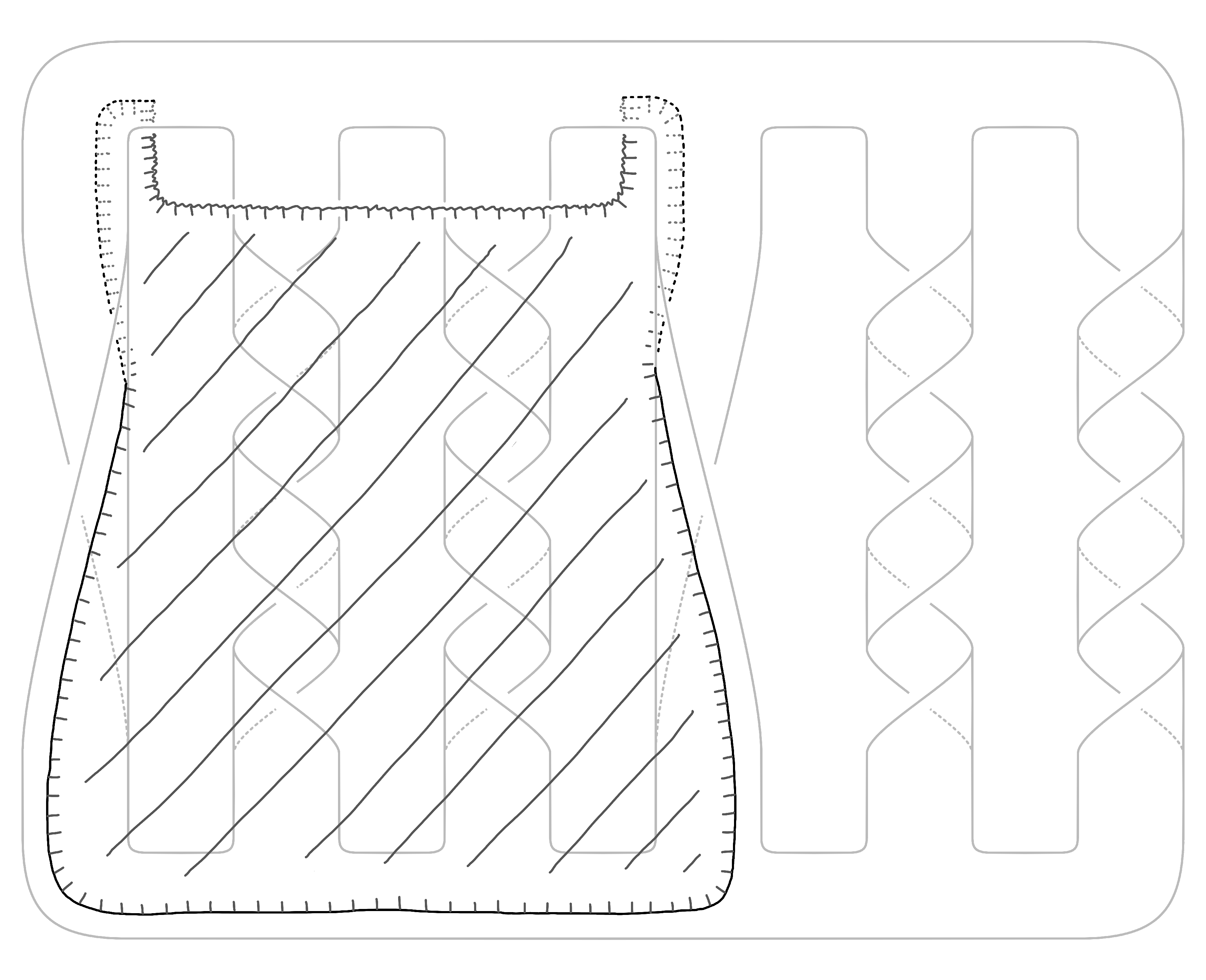} \\
 \includegraphics[scale=0.15]{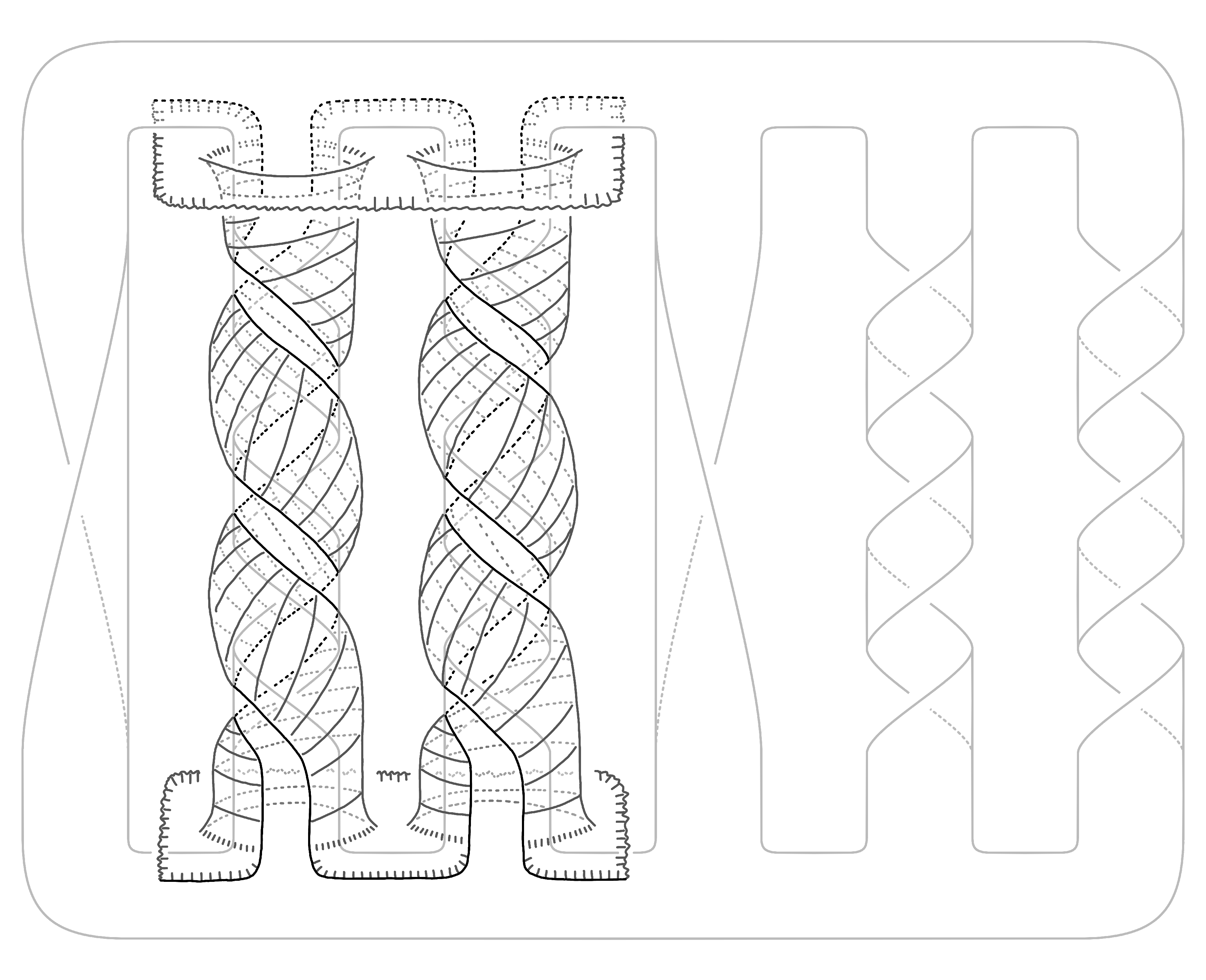} \\
 \includegraphics[scale=0.15]{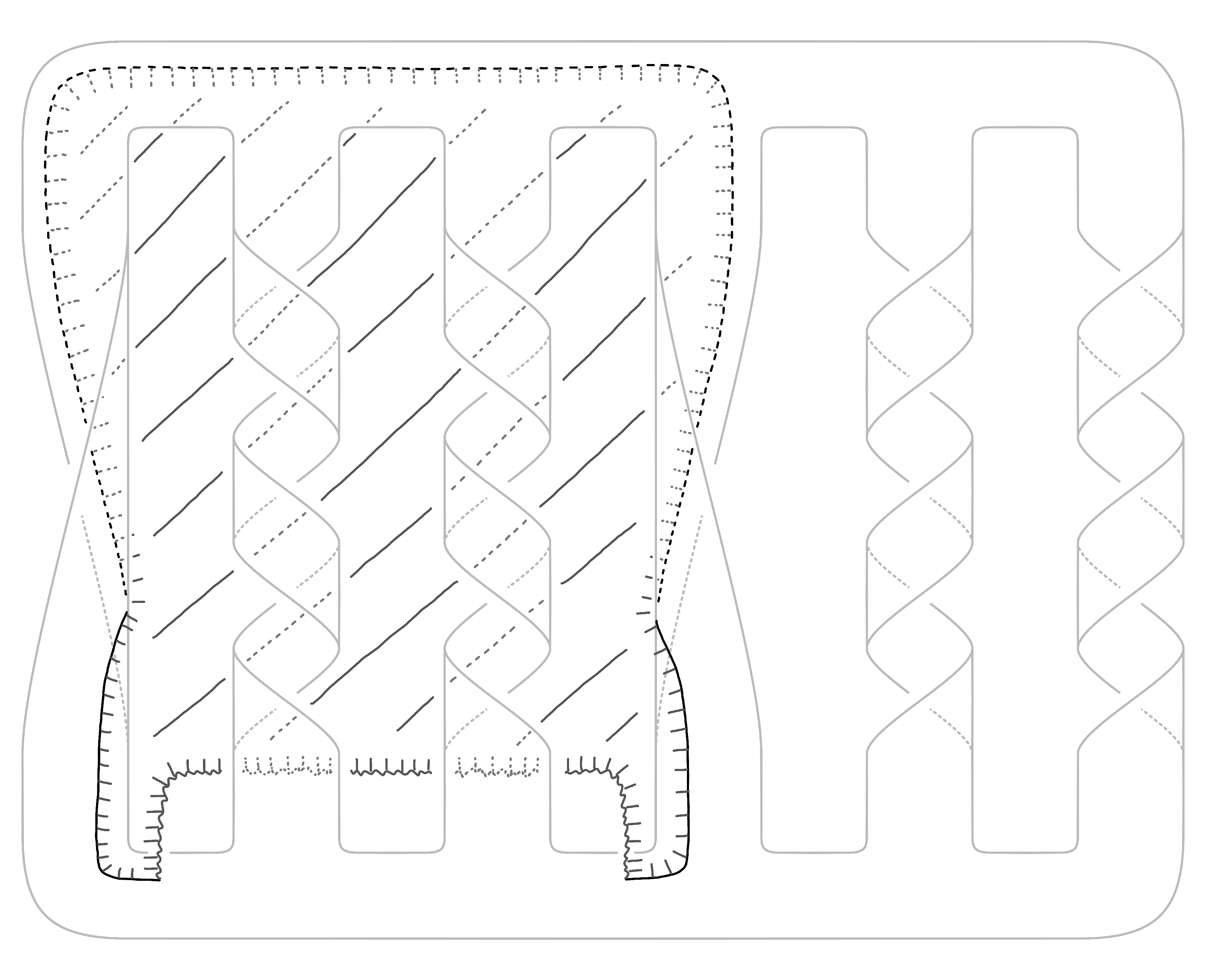}
 \caption{An altering surface for the Seifert surface depicted in Figure \ref{fig:1_m5_m5_m1_5_5_orig}, which is decomposed into three peaces along jagged lines for clarity}
 \label{fig:1_m5_m5_m1_5_5_altering_surface}
\end{figure}

In general, for some integers $m, n \geq 2$, let $k_{1}, k_{2}, \dots, k_{m}$ and $l_{1}, l_{2}, \dots, l_{n}$ be integers whose absolute values are greater than or equal to $5$.
Consider a Seifert surface of the $(1, k_{1}, k_{2}, \dots, k_{m}, -1, l_{1}, l_{2}, \dots, l_{n})$-pretzel link in a similar way to Figure \ref{fig:1_m5_m5_m1_5_5_orig}.
Then we can check along the same line to \cite{Oer1980, Par1978} that this Seifert surface is not minimal but incompressible.
On the other hand, we have an altering surface of the Seifert surface in a similar way to Figure \ref{fig:1_m5_m5_m1_5_5_altering_surface} which yields a minimal Seifert surface of the link.
We note that this altering surface is homeomorphic to the twice punctured genus $\frac{m - 2}{2}$ surface if $m$ is even, otherwise the once punctured genus $\frac{m - 1}{2}$ surface.
\end{example}

\begin{example}
It is easy to see that the SAS consisting of the four mutually disjoint altering surfaces, depicted in Figure \ref{fig:trivial_knot_SAS}, for Kobayashi's Seifert surface of the trivial knot, mentioned in Section \ref{sec:introduction}, yields a minimal Seifert surface (i.e., a disk) of the trivial knot.
\begin{figure}[htbp]
 \centering
 \includegraphics[scale=0.20]{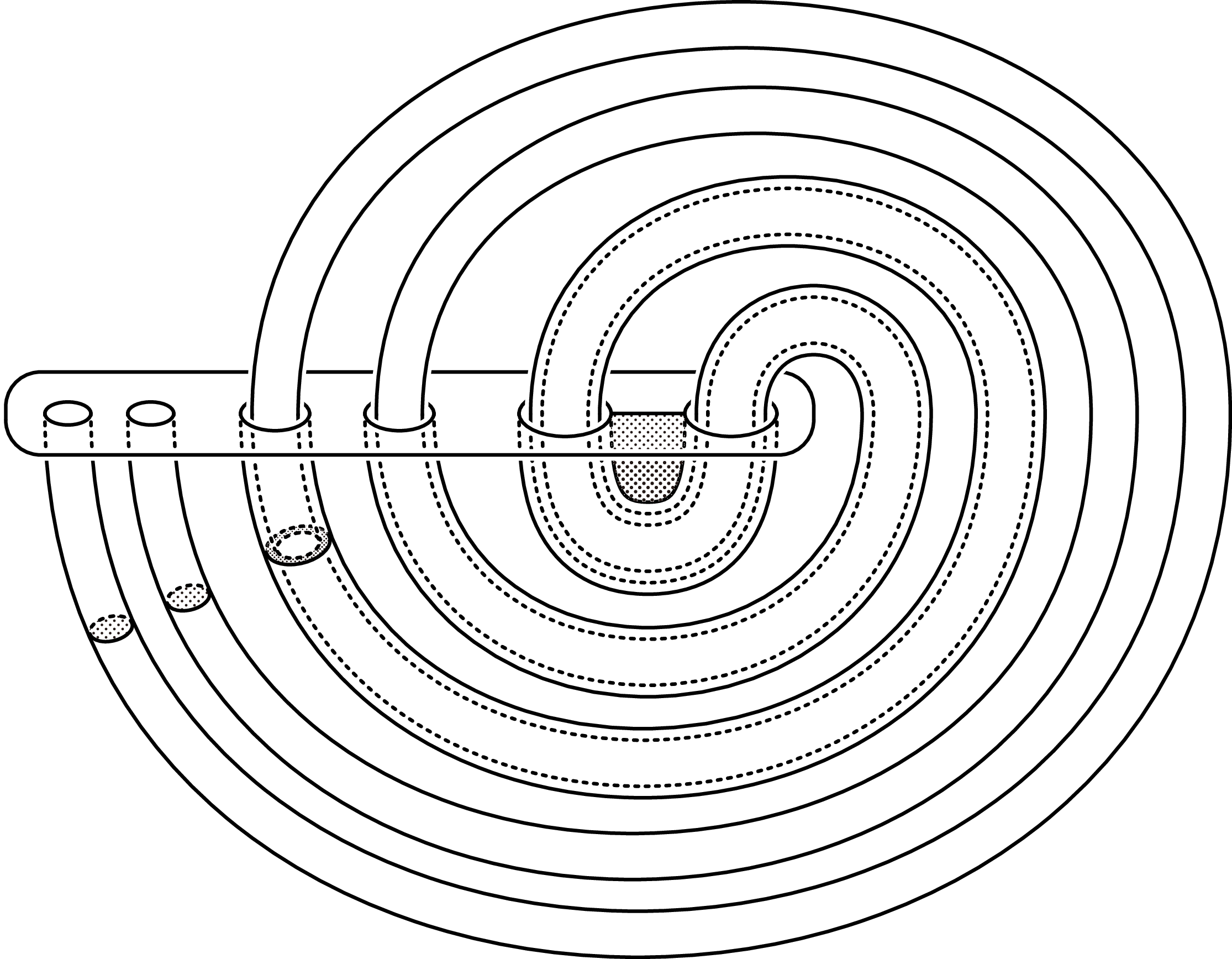}
 \caption{Four mutually disjoint altering surfaces (dotted ones) for Kobayashi's Seifert surface of the trivial knot}
 \label{fig:trivial_knot_SAS}
\end{figure}
\end{example}

\section*{Acknowledgments}
The author wishes to express his gratitude to Professor Mikami Hirasawa for invaluable conversations.
He is partially supported by JSPS KAKENHI Grant Number JP19K03476.

\bibliographystyle{amsplain}

\end{document}